\documentclass[a4paper,12pt]{article}
\usepackage[utf8]{inputenc}
\usepackage[english]{babel}
\usepackage{csquotes} 
\usepackage{amsfonts} 
\usepackage{amsmath} 
\usepackage{amsthm} 
\usepackage{mathtools}
\usepackage{enumitem}
\usepackage[normalem]{ulem} 

\usepackage{graphicx} 
\graphicspath{ {Immagini/} } 
\usepackage{wrapfig} 
\usepackage{subcaption} 
\usepackage[font={footnotesize,it}]{caption} 
\usepackage{hyperref} 
\usepackage[width=150mm,top=25mm,bottom=25mm,headheight=14.5pt]{geometry}
\usepackage{todonotes}
\usepackage[backend=biber,style=numeric,useprefix]{biblatex}
\theoremstyle{plain}
\newtheorem{theorem}{Theorem}[section]

\newtheorem{lemma}{Lemma}[section]

\newtheorem*{observation}{Remark}

\title{Linear error bounds for HJB equations in finite horizon control problems.}
\author{Alessandro Alla\thanks{Sapienza, Università di Roma, Italy \texttt{\{alessandro.alla, filippo.mayer\}@uniroma1.it}} \and Filippo Mayer\footnotemark[1]}
\date{February 2026}

\begin{document}

\maketitle

\abstract{%
We study semi-Lagrangian approximation schemes for Hamilton–Jacobi–Bel\-lman equations arising from finite horizon optimal control problems. Classical error estimates for these schemes include the term $\frac{1}{\Delta t}$ which leads to pessimistic convergence bounds and is not observed in numerical experiments. In this work, we provide improved error estimates under standard regularity assumptions on the dynamics, the running cost, and the final cost, assuming the presence of a positive discount factor. The new bound depends linearly on the time step, the spatial mesh size, and a measure of the temporal oscillation of the control, thus removing the mixed term appearing in previous analyses. The proof relies on a refined comparison between continuous and discrete cost functionals and on stability estimates for the controlled dynamics. Numerical experiments confirm first-order convergence in both space and time and suggest that the improved behavior persists even in the undiscounted case.
}

\section{Introduction}

Hamilton–Jacobi–Bellman (HJB, \cite{bardicapuzzodolcetta}) equations play a central role in optimal control theory, providing a characterization of the value function through the Dynamic Programming Principle (DPP, \cite{bellman}). Due to their nonlinearity, however, in most practical situations explicit solutions of HJB equations are not available, and numerical approximation schemes are required. Among these, semi-Lagrangian (SL, \cite{F,Fec,falconeferretti}) schemes have proven to be particularly effective, thanks to their stability properties, their ability to handle general nonlinear dynamics, and the absence of restrictive CFL conditions. Finally, this approach provides the optimal control in feedback form which is fundamental in applications due to its robustness with respect to noise and error measurements.

For finite horizon optimal control problems, the semi-Lagrangian approximation introduced in the seminal work \cite{falconegiorgi} provides a consistent and monotone discretization of the value function showing an order of convergence of order $O(\frac{\Delta x}{\Delta t} + \Delta t)$ provided that $\Delta x$ and $\Delta t$ are the spatial and temporal step size respectively.

The order of convergence found is classical for SL schemes, even if numerical results usually show more accurate approximation than the theoretical bounds. Recently, sharper error estimates have been obtained for infinite horizon discounted control problems in \cite{DN}, showing that the semi-Lagrangian approximation converges with first order in both space and time, without the appearance of the term $\frac{1}{\Delta t}$. These results rely on the contractive structure induced by the discount factor and an appropriate discretization of the cost functional which helps us to provide the bounds otherwise difficult to obtain. However, performing this analysis to finite horizon problems is more complicated, since the direct dependence on time of the value function has to be carefully taken into account in order to get analogous results.

The aim of this work is to improve the error estimates studied in \cite{falconegiorgi} for SL schemes within the framework of finite horizon optimal control problems adapting the approach proposed in \cite{DN}. Under standard regularity assumptions on the dynamics, running cost, and final cost, and assuming a positive discount factor, we prove that the numerical value function satisfies the error bound $O(\Delta t + \Delta x + M_u)$, where $M_u$ goes to zero if regularity assumptions on the controls are assumed (more details will be discussed in Section \ref{sec:3}).


The proof relies on an error bound on the cost functional provided an appropriate discrete version of it. To obtain this bound, we exploit the regularity assumptions on the dynamics, the running cost and the final cost and standard tools used in differential equation analysis such as Gronwall's lemma.

Finally, we validate the theoretical results through numerical experiments. In particular, we consider a nonlinear control problem with an explicit value function, for which the experimental order of convergence confirms first-order accuracy in both space and time. Additional tests suggest that the improved convergence behavior persists even in the undiscounted case, indicating possible directions for future extensions of the theory.\\

The remainder of this paper is structured as follows. Section \ref{sec:2} recalls the optimal control problem, its discretization and the error estimates provided in \cite{falconegiorgi}. Section \ref{sec:3} contains the core of this paper. We provide error analysis for the discretized finite horizon control problem. Numerical results to confirm our theoretical findings are provided in Section \ref{sec:4}. Finally, Section \ref{sec:5} summarizes the key findings and outlines potential directions for future research.

\section{Optimal control problem}\label{sec:2}

In this section, we introduce our optimal control within the finite horizon framework. Let us consider a controlled ordinary differential equation (ODE) of the form
\begin{equation} \label{eq:EDO_control}
\left\{
\begin{aligned}
\dot{y}(t) &= f(y(t),t,u(t)), \qquad t \in (t_0,T) \\
y(t_0) &= x ,
\end{aligned}
\right.
\end{equation}
where the dynamics $f : \mathbb{R}^d \times (t_0,T) \times U \rightarrow \mathbb{R}^d$ is bounded, continuous and Lipschitz with respect to the first variable uniformly with respect to the other two, the control $u : (t_0,T) \rightarrow U \subset \mathbb{R}^c$ is measurable and the initial condition $x \in \mathbb{R}^d$. We will denote $\mathcal{U}:= \{u : (t_0,T) \rightarrow U\ :\ u\ \mathrm{measurable}\}$ the space of admissible controls.We remark that under these assumptions we can prove the existence and uniqueness of a solution for (\ref{eq:EDO_control}).

The cost functional we consider is given by
\begin{equation}\label{eq:cost}
J(x,t,u) = \int_{t}^T g(y(s),s,u(s))e^{-\lambda (s-t)} ds + e^{-\lambda (T-t)} \psi(y(T)),
\end{equation}
where $y : (t_0,T) \rightarrow \mathbb{R}^d$ is the solution of (\ref{eq:EDO_control}), $g : \mathbb{R}^d \times (t_0,T) \times U \rightarrow \mathbb{R}$ is the running cost, which is typically assumed to be bounded, continuous and Lipschitz with respect to the first variable uniformly with respect to the other two, and $\psi: \mathbb{R}^d \rightarrow \mathbb{R}$ is the final cost which is bounded and uniformly continuous.

The value function reads
\[
v(x,t) = \inf_{u \in \mathcal{U}} J(x,t,u).
\]

The characterization of the value function through the Dynamic Program\-ming Principle (DPP) for finite horizon problems is $\forall x \in \mathbb{R}^d$, $\forall \tau \in [t,T]$
\begin{equation}\label{eq:DPP}
v(x,t) = \inf_{u \in \mathcal{U}} \left\{ \int_t^{\tau} g(y(s),s,u(s)) e^{-\lambda (s-t)} ds + e^{-\lambda(\tau-t)} v(y(\tau),\tau)\right\}.
\end{equation}

The correspondent Hamilton-Jacobi-Bellman (HJB) equation reads
\begin{equation}\label{eq:HJB}
\left\{
\begin{aligned}
-v_t(x,t) + \lambda v(x,t) + H(x,t,\nabla v) &= 0 \\
v(x,T) &= \psi(x),
\end{aligned}
\right.
\end{equation}
where $H(x,\nabla v) = \sup_{u \in U} \left\{ -f(x,t,u) \cdot \nabla v(x,t) - g(x,t,u)\right\}$.

 Existence and uniqueness of the HJB equations has been proved for a particular class of weak solutions called viscosity solutions (see \cite{CL}). Usually, it is not possible to compute an analytical solution of \eqref{eq:HJB}, therefore, in what follows, we discuss the semi-Lagrangian discretization for \eqref{eq:DPP}. We first provide the semi-discretization on time and then the fully discretization.

\paragraph{Time discretization.}
We perform a first discretization in time for both the value function and the cost functional. We fix $N_t \in \mathbb{N}$ and consider a temporal grid $\{t_0,...,t_{N_t}\}$, where $t_n = t_0 + n \Delta t$, $\Delta t = \frac{T-t_0}{N_t}$. The semidiscrete cost functional at time $t_n$ is 
\[
J_{\Delta t}^{t_n}(x,\uline{u}) = \Delta t \sum_{k=n}^{N_t-1}g(y_k,t_k,u_k) (1-\lambda \Delta t)^{k-n} + (1-\lambda \Delta t)^{N_t-n} \psi(y_{N_t}),
\]
where $y_n\approx y(t_n)$ is defined by the explicit Euler approximation of \eqref{eq:EDO_control}
\begin{equation} \label{eq:EDO_discrete}
\left\{
\begin{aligned}
y_{n+1} &= y_n + \Delta t f(y_n,t_n,u_n), \qquad n=0,...,N_t-1 \\
y_0 &= x,
\end{aligned}
\right.
\end{equation}
and $\uline{u} = (u_0,...,u_{N_t}) \in U^{N_t+1}$ is a sequence of elements of $U$. In our setting we will assume that $\uline{u}$ approximates a control $u \in \mathcal{U}$ at the discrete times $t_0,...,t_{N_t}$, i.e. $u_n \approx u(t_n)$ for any $ n=0,...,N_t$.

The semidiscrete approximation of the value function, $v_{\Delta t}(\cdot,t_n)\approx v(\cdot,t_n)$, for each $x \in \mathbb{R}^d$ reads
\begin{equation*}
\left\{
\begin{aligned}
v_{\Delta t}(x,t_n) &= \inf_{u \in U} \left\{ \Delta t g(x,t_n,u) + (1-\lambda \Delta t) v_{\Delta t}(x+\Delta t f(x,t_n,u),t_{n+1}) \right\}& \\
&\phantom{\inf_{u \in U} \{ \Delta t g(x,u) + (1-\lambda \Delta t) v_{\Delta t}} \qquad \qquad \qquad n=N_t-1,...,0 \\
v_{\Delta t}(x,T) &= \psi(x).
\end{aligned}
\right.
\end{equation*}

This semi-discretization comes from a first order approximation of equivalently (\ref{eq:DPP}) or (\ref{eq:HJB}).

\paragraph{Space Discretization.} For the space discretization we suppose that our equation is posed on a polytope $\Omega \subset \mathbb{R}^d$ such that
\begin{equation} \label{eq:Omega_hp}
x + \Delta t f(x,t,u) \in \overline{\Omega} \qquad \forall x \in \overline{\Omega}, \  \forall t \in (t_0,T),\ \forall u \in U,
\end{equation}
that we discretize with a mesh of polyhedra $\{S_i\}_i$ whose nodes will be denoted by $\{x_i\}_{i=1,...,N_x}$. Its  diameter\footnote{Here the diameter is the maximum over all the elements of the mesh of the maximum distance between two points of the same element.} will be denoted by $\Delta x$.

The discrete approximation of the value function, $V(x_i,t_n)\approx v(x_i,t_n)$, is defined for each node $x_i$ of the mesh by
\begin{equation} \label{eq:SL}
\left\{
\begin{aligned}
V(x_i,t_n) &= \inf_{u \in U} \left\{ \Delta t g(x_i,t_n,u) + (1-\lambda \Delta t) \textbf{I}_1[V](x_i+\Delta t f(x_i,t_n,u),t_{n+1}) \right\}& \\ & \phantom{\inf_{u \in U} \{ \Delta t g(x_i,u) + (1-\lambda \Delta t)} \qquad n=0,...,N_t-1,\ i=1,...,N_x \\
V(x_i,T) &= \psi(x_i).
\end{aligned}
\right.
\end{equation}
\paragraph{First error estimate.}
Error estimates for the numerical scheme \eqref{eq:SL} have been studied in \cite{falconegiorgi}. The main result is recalled in the theorem below.

\begin{theorem}
Assuming that the functions $f$ and $g$ are continuous, and that $f$, $g$ and $\psi$ are uniformly Lipschitz with respect to the space variable and bounded. Then we have
\begin{equation}\label{err_est_fg}
\|v-V\|_{\infty} \leq \tilde{C} \left( \frac{\Delta x}{\sqrt{\Delta t}} + \sqrt{\Delta t} \right)
\end{equation}
for some positive constant $\tilde{C}\in\mathbb{R}$.
\end{theorem}

This estimate features the term $\frac{1}{\Delta t}$ which is not observed in numerical experiments. Indeed, the goal of this work is to improve the estimate \eqref{err_est_fg}.

\section{Error estimates}\label{sec:3}

This section provides a new error estimate for the numerical value function $V$ approximated using \eqref{eq:SL}. We extend the results provided in \cite{DN} for the infinite horizon case to our setting described in Section \ref{sec:2}. Our goal is to improve \eqref{err_est_fg} with a bound depending linearly only on $\Delta t$ and $\Delta x$.

We will assume the following hypotheses:
\begin{enumerate}[label={(H\textsubscript{\arabic*})}]
\item \label{hp:1} $f$ is continuous, Lipschitz with respect to the first variable, uniformly with respect to the second and the third variable, with corresponding Lipschitz constant $L_f$:
\[
\|f(x,t,u) - f(y,t,u)\|_{\infty} \leq L_f \|x-y\|_{\infty} \quad \forall x,y \in \mathbb{R}^d,\ \forall t \in (t_0,T),\ \forall u \in \mathcal{U},
\]
Lipschitz with respect to the second variable, uniformly with respect to the first and third variable, with Lipschitz constant $L_f$:
\[
\|f(x,t,u) - f(x,s,u)\|_{\infty} \leq L_f \lvert t-s\rvert \quad \forall x \in \mathbb{R}^d,\ \forall t,s \in (t_0,T),\ \forall u \in \mathcal{U},
\]
Lipschitz with respect to the third variable, uniformly with respect to the first and second variable, with Lipschitz constant $L_f$:
\[
\|f(x,t,u) - f(x,t,w)\|_{\infty} \leq L_f \|u-w\|_{\infty} \quad \forall x \in \mathbb{R}^d,\ \forall t \in (t_0,T),\ \forall u,w \in \mathcal{U},
\]
and bounded by $M_f$:
\[
\|f(x,t,u)\|_{\infty} \leq M_f \qquad \forall x \in \mathbb{R}^d,\ \forall t \in (t_0,T),\ \forall u \in \mathcal{U};
\]
\item \label{hp:2} $g$ is continuous, Lipschitz with respect to the first variable, uniformly with respect to the first and second with Lipschitz constant $L_g$:
\[
\lvert g(x,t,u) - g(y,t,u)\rvert \leq L_g \|x-y\|_{\infty} \quad \forall x,y \in \mathbb{R}^d,\ \forall t \in (t_0,T),\ \forall u \in \mathcal{U},
\]
Lipschitz with respect to the second variable, uniformly with respect to the first and third variable, with Lipschitz constant $L_f$:
\[
\lvert g(x,t,u) - g(x,s,u)\rvert \leq L_g \lvert t-s \rvert \quad \forall x \in \mathbb{R}^d,\ \forall t,s \in (t_0,T),\ \forall u \in \mathcal{U},
\]
Lipschitz with respect to the third variable, uniformly with respect to the first and second variable, with Lipschitz constant $L_g$:
\[
\lvert g(x,t,u) - g(x,t,w)\rvert \leq L_g \|u-w\|_{\infty} \quad \forall x \in \mathbb{R}^d,\ \forall t \in (t_0,T),\ \forall u,w \in \mathcal{U},
\]
and bounded by $M_g$:
\[
\lvert g(x,t,u)\rvert \leq M_g \qquad \forall x \in \mathbb{R}^d,\ \forall t \in (t_0,T),\ \forall u \in \mathcal{U}.
\]
\item \label{hp:3} $\psi$ is Lipschitz:
\[
\lvert \psi(x) - \psi(y) \rvert \leq L_{\psi} \|x-y\|_{\infty} \qquad \forall x,y \in \mathbb{R}^d
\]
and bounded:
\[
\lvert \psi(x) \rvert \leq M_{\psi} \qquad \forall x \in \mathbb{R}^d.
\]
\end{enumerate}

\begin{observation}
Unlike \cite{falconegiorgi}, we further assume the uniform Lipschitz continuity for the functions $f$ and $g$ with respect to the controls (as also done in \cite{DN}).
\end{observation}

In what follows we will use a fully-discrete version of the cost functional, which is defined for any $n=0,...,N_t$ as
\begin{equation} \label{eq:cost_discrete}
J^{t_n}_{\Delta t,\Delta x}(x,\uline{u}) = \Delta t \sum_{k=n}^{N_t-1} \textbf{I}_1[g](y_k,t_k,u_k) (1-\lambda \Delta t)^{k-n} + (1-\lambda \Delta t)^{N_t-n} \psi(y_{N_t}),
\end{equation}
where $\textbf{I}_1[\cdot]$ is the $d$-linear interpolation operator defined on the mesh.

We will denote by
\[
M_u := \max_{n=0,...,N_t-1} \max_{s \in [t_n,t_{n+1})} \|u(s)-u(t_n)\|_{\infty}
\]
the maximum distance between two values of the same control over one time interval of our grid. We note that $M_u$ can be bounded with the time step provided that the controls are regular (e.g. Lipschitz continuity).

The following result establishes a further relation between $v$ and its discrete counterpart $V$, since here we prove that $V(x,t_n)$ can be seen as the infimum over all $\uline{u} \in U^{N_t+1}$ of the discrete functional $J_{\Delta t,\Delta x}^{t_n}(x,\uline{u})$ for all $x \in \overline \Omega$ and $n=0,...,N_t-1$, in analogy with the definition of $v$ as the infimum over all the controls of the cost functional.

\begin{observation}
We note that in this discrete setting the infimum is pointwise, indeed it depends on the initial condition $x \in \overline \Omega$ where we compute the value function, i.e.
\[
\uline{u}^*:= \underset{\uline{u} \in U^{N_t-n+1}}{\mathrm{arginf}} J_{\Delta t,\Delta x}(x,\uline{u}) = \uline{u}^*(x).
\]
\end{observation}

\begin{theorem} \label{th:1}
Assume \ref{hp:3} holds true. The function $V(\cdot,t_n) \in \mathcal{W}$,
\[
\mathcal{W}:= \{ w : \Omega \rightarrow \mathbb{R}\mid w \in C(\Omega),\ \nabla w(x) = c_i \in \mathbb{R}^d,\ \forall x \in S_i,\ \forall i=1,...,N_x\},
\]
solution to (\ref{eq:SL}) for all $ n=0,...,N_t$, satisfies the following equality
\[
V(x,t_n) = \inf_{\uline{u} \in U^{N_t-n+1}} J^{t_n}_{\Delta t, \Delta x}(x,\uline{u}) \qquad \forall x \in \overline{\Omega},\ \forall n=0,...,N_t,
\]
where $\uline{u} = (u_n,...,u_{N_t})$.
\end{theorem}

\begin{proof}
Let us first note that $\textbf{I}_1[g]$ inherits the same boundedness properties as $g$ due to its linearity.
Thus, we have
\[
(1-\lambda \Delta t)^n \textbf{I}_1[g](y_n,t_n,u_n) \leq (1-\lambda \Delta t)^n M_g, \qquad \forall n=0,...,N_t-1,
\]
which proves the boundedness of $J^{t_n}_{\Delta t,\Delta x}$ for all $n=0,...,N_t-1$. Furthermore, for all $n=0,...,N_t$ and for any $\uline{u} \in U^{N_t-n+1}$,
\begin{equation}\label{eqW}
J^{t_n}_{\Delta t, \Delta x} \in \mathcal{W},
\end{equation}
since, by the definition of $\textbf{I}_1$, we have
\[
\textbf{I}_1[g](y_n,t_n,u_n) \in \mathcal{W} \qquad \forall n=0,...,N_t.
\]

Let us define for each $n=0,...,N_t$ the auxiliary function
\[
w^{t_n}(x) := \inf_{\uline{u} \in U^{N_t-n+1}} J^{t_n}_{\Delta t,\Delta x}(x,\uline{u}).
\]
From \eqref{eqW}, it follows that $w^{t_n} \in \mathcal{W}$, for all $n=0,...,N_t$. We want to prove $w^{t_n}(x) = V(x,t_n)$ for all $x \in \overline{\Omega}$ and for all $n=0,...,N_t$.

Let us fix $i \in \{1,...,N_x\}, n \in \{0,...,N_t-1\},\ \uline{u} = (u_n,...,u_{N_t}) \in U^{N_t-n+1}$ and set $\overline{u} := (u_{n+1},..., u_{N_t}) \in U^{N_t-n}$ the sequence of all the elements of $\uline{u}$ except $u_n$. Then, since $\textbf{I}_1[g]$ coincides with $g$ on the nodes of the space mesh,
\begin{align*}
J^{t_n}_{\Delta t,\Delta x}&(x_i,\uline{u}) = \Delta t \textbf{I}_1[g](x_i,t_n,u_n) +\Delta t \sum_{k=n+1}^{N_t-1} (1-\lambda \Delta t)^k \textbf{I}_1[g](y_k,t_k,u_k)^{k-n} + \\
&+ (1-\lambda \Delta t)^{N_t-n} \psi(y_{N_t}) = \\
&= \Delta t g(x_i,t_n,u_n) + \Delta t (1-\lambda \Delta t) \sum_{k=n+1}^{N_t-1} (1-\lambda \Delta t)^{k-(n+1)} \textbf{I}_1[g](y_{k},t_{k},u_{k}) + \\
&+ (1-\lambda \Delta t)^{N_t-n} \psi(y_{N_t}) =\\
&= \Delta t g(x_i,t_n,u_n) + (1-\lambda \Delta t) J^{t_{n+1}}_{\Delta t,\Delta x} (x_i + \Delta t f(x_i,t_n,u_n),\overline{u}).
\end{align*}
By the definition of $w^{t_n}$
\[
J_{\Delta t,\Delta x}^{t_n}(x_i,\uline{u}) \geq \Delta t g(x_i,t_n,u_n) + (1-\lambda \Delta t) w^{t_{n+1}}(x_i + \Delta t f(x_i,t_n,u_n)),
\]
from which
\begin{equation} \label{eq:proof1}
\begin{aligned}
w^{t_n}(x_i) &= \inf_{\uline{u} \in U^{N_t-n+1}} J^{t_n}_{\Delta t,\Delta x}(x_i,\uline{u})\geq \\
&\geq \inf_{\uline{u} \in U^{N-t-n+1}} \left\{ \Delta t g(x_i,t_n,u_n) + (1-\lambda \Delta t) w^{t_{n+1}}(x_i + \Delta t f(x_i,t_n,u_n)) \right\} = \\
&= \inf_{u \in U} \left\{ \Delta t g(x_i,t_n,u) + (1-\lambda \Delta t) w^{t_{n+1}}(x_i + \Delta t f(x_i,t_n,u)) \right\},
\end{aligned}
\end{equation}
where the last equality is reached since  we are considering only one element of $\uline{u}$.
For $\varepsilon>0$, let us consider $u^{\varepsilon}_n \in U$, and let $z:=x_i+\Delta t f(x_i,t_n,u_n^{\varepsilon})$. From the definition of $w^{t_n}$, there exists $\uline{u}^{\varepsilon} = (u_{n+1}^{\varepsilon},...,u_{N_t}^{\varepsilon}) \in U^{N_t-n}$ such that
\[
w^{t_n}+\varepsilon \geq J^{t_n}_{\Delta t,\Delta x}(z,\uline{u}^{\varepsilon}).
\]
Let $\tilde{u}:=(u_n^{\varepsilon},u_{n+1}^{\varepsilon},...,u_{N_t}^{\varepsilon}) \in U^{N_t+1}$. Repeating the computations above, we get
\begin{align*}
J^{t_n}_{\Delta t,\Delta t}(x_i,\tilde{u}) &= \Delta t g(x_i,t_n,u_n^{\varepsilon}) + (1-\lambda \Delta t) J^{t_{n+1}}_{\Delta t,\Delta x}(z,\uline{u}^{\varepsilon}) \leq \\
& \leq \Delta t g(x_i,t_n,u_n^{\varepsilon}) + (1-\lambda \Delta t) (w^{t_{n+1}}(z)+\varepsilon).
\end{align*}
Then, due to the arbitrariness of $u_n^{\varepsilon}$, we have
\[
w^{t_n}(x_i) \leq \inf_{u \in U} \left\{ \Delta t g(x_i,t_n,u) + (1-\lambda \Delta t) w^{t_{n+1}}(z) \right\} + (1-\lambda \Delta t) \varepsilon,
\]
and for the arbitrariness of $\varepsilon$,
\begin{equation} \label{eq:proof2}
w^{t_n}(x_i) \leq \inf_{u \in U} \left\{ \Delta t g(x_i,t_n,u) + (1-\lambda \Delta t) w^{t_{n+1}}(z) \right\}.
\end{equation}
By putting (\ref{eq:proof1}) and (\ref{eq:proof2}) together, we get
\[
w^{t_n}(x_i) = \inf_{u \in U} \left\{ \Delta t g(x_i,t_n,u) + (1-\lambda \Delta t) w^{t_{n+1}}(z) \right\}.
\]
This relation defines a solution of (\ref{eq:SL}) at the time $t_n$.

Defining $w^{t_{N_t}}(x_i):=\psi(x_i)$ for all $i=1,...,N_x$, we get that $(w^{t_0},...,w^{t_{N_t}}) \in \mathcal{W}^{N_t+1}$ is a solution of (\ref{eq:SL}), and so we can conclude
\[
w^{t_n}(x) = V(x,t_n) \qquad \forall x \in \overline{\Omega},\ \forall n=0,...,N_t.
\]
\end{proof}

\begin{lemma} \label{th:lemma1}
Assume \ref{hp:1} and \ref{hp:2} hold true. Let $y(t)$ be the solution of (\ref{eq:EDO_control}) and $\overline{y}(t)$ the piecewise constant extension of $(y_0,..., y_{N_t})$ the sequence defined in (\ref{eq:EDO_discrete}). Then, for each $t \in (t_0,T)$ the following inequality holds:
\begin{align*}
\|y(t) - \overline{y}(t)\|_{\infty} &\leq e^{L_I (t-t_0)} \left( (t-t_0) L_f (M_u + \Delta t + \Delta x) + 2 M_f \Delta t\right) \leq \\
&\leq e^{L_I t} \left( t L_f (M_u + \Delta t + \Delta x) + 2 M_f \Delta t\right),
\end{align*}
where $L_I:= d C L_f$ and $C\in\mathbb{R}$ is a positive constant. 
\end{lemma}

\begin{proof}

Let us define for each $t \in (t_0,T)$
\[
\overline{u}(t):= \sum_{n=0}^{N_t-1} u_n \chi_{[t_n,t_{n+1})}(t) \qquad \mathrm{and} \qquad
\overline{t}(t):= \sum_{n=0}^{N_t-1} t_n \chi_{[t_n,t_{n+1})}(t)
\]
the piecewise constant extensions respectively of $(u_0,...,u_{N_t})$ and $(t_0,...,t_{N_t})$ on the interval $[t_0,T]$. Then, since $f$ is constant on the subintervals of our time grid, we have
\[
\overline{y}(t) = x + \int_{t_0}^{\left\lfloor\frac{t}{\Delta t}\right\rfloor \Delta t} \textbf{I}_1[f](\overline{y}(s),\overline{t}(s),\overline{u}(s)) ds,
\]
where $\lfloor\cdot\rfloor$ denotes the floor function, and we can estimate the difference between $y$ and $\overline{y}$ as
\begin{align*}
y(t) - \overline{y}(t) &= \int_{t_0}^{\left\lfloor\frac{t}{\Delta t}\right\rfloor \Delta t} \left( f(y(s),s,u(s)) - \textbf{I}_1[f](\overline{y}(s),\overline{t}(s),\overline{u}(s)) \right) ds + \\
&+ \int_{\left\lfloor\frac{t}{\Delta t}\right\rfloor \Delta t}^t f(y(s),s,u(s)) ds .
\end{align*}
Taking the $\|\cdot\|_\infty$ norm and using the boundedness of $f$ yields
\[
\|y(t) - \overline{y}(t)\|_{\infty} \le \int_{t_0}^{\left\lfloor\frac{t}{\Delta t}\right\rfloor \Delta t} \| f(y(s),s,u(s)) - \textbf{I}_1[f](\overline{y}(s),\overline{t}(s),\overline{u}(s)) \|_{\infty} ds + 2 M_f \Delta t .
\]

\medskip
\noindent
\textbf{Step 1. Decomposition of the integrand.}
For $s\in[t_0,T]$ we decompose
\begin{align}\label{stimaf}
\begin{aligned}
\| f(y(s),s,u(s)) &- \textbf{I}_1[f](\overline{y}(s),s,\overline{u}(s)) \|_{\infty} \leq \|f(y(s),s,u(s)) - f(y(s),\overline{t}(s),\overline{u}(s))\|_{\infty} + \\
&+ \|f(y(s),\overline{t}(s),\overline{u}(s)) - \textbf{I}_1[f](y(s),\overline{t}(s),\overline{u}(s))\|_{\infty} + \\
&+ \|\textbf{I}_1[f](y(s),\overline{t}(s),\overline{u}(s)) - \textbf{I}_1[f](\overline{y}(s),s,\overline{u}(s))\|_{\infty} .
\end{aligned}
\end{align}
The first term accounts for time and control discretization and, by the
Lipschitz continuity of $f$, satisfies 
\begin{align*}
\|f(y(s),s,u(s)) - f(y(s),\overline{t}(s),\overline{u}(s))\|_{\infty} &\leq \|f(y(s),s,u(s)) - f(y(s),s,\overline{u}(s))\|_{\infty} + \\
&+ \|f(y(s),s,\overline{u}(s)) - f(y(s),\overline{t}(s),\overline{u}(s)\|_{\infty} \leq \\
&\leq L_f \left( \|u(s) - \overline{u}(s)\|_{\infty} + \lvert s-\overline{t}(s) \rvert \right) = \\
&= L_f \left( \|u(s) - u(t_p)\|_{\infty} + \lvert s-\overline{t}(s) \rvert \right) \leq \\
&\leq L_f (M_u + \Delta t),
\end{align*}
for some $p \in \{0,...,N_t\}$.

The second term in the right part of \eqref{stimaf} corresponds to the interpolation error (see e.g. \cite{falconeferretti}) and can be estimated as
\[
\|f(y(s), s, u(s))
-\textbf{I}_1[f](y(s),\overline t(s),\overline u(s))\|_\infty
\le L_f\,\Delta x .
\]


Finally, using the linearity of the interpolation operator $\textbf{I}_1$, the Lipschitz continuity of $f$ and Cauchy-Schwarz inequality, it can be proved that $\textbf{I}_1[f]$ is Lipschitz with constant $d C L_f$, where the multiplying $d$ is due to the equivalence of norms in finite-dimensional spaces, and the constant $C>0$ arises from the bound of the gradient of $\textbf{I}_1[f]$. Thus, we can write
\begin{align*}
\|\textbf{I}_1[f](y(s),\overline{t}(s),\overline{u}(s)) -& \textbf{I}_1[f](\overline{y}(s),\overline{t}(s),\overline{u}(s))\|_{\infty} \leq d C L_f \|y(s) - \overline{y}(s)\|_{\infty}.
\end{align*}

\medskip
\noindent
\textbf{Step 2. Integral inequality and Gronwall argument.}
Collecting the previous estimates{, we get}
\begin{align*}
\|y&(t) - \overline{y}(t)\|_{\infty} \leq \\
&\leq \int_{t_0}^{\left\lfloor\frac{t}{\Delta t}\right\rfloor \Delta t} \left( L_f (M_u + \Delta t) + L_f \Delta x + d C L_f \|y(s) - \overline{y}(t)\|_{\infty} \right) ds + 2 M_f \Delta t = \\
&=  d C L_f \int_{t_0}^{\left\lfloor\frac{t}{\Delta t}\right\rfloor \Delta t} \|y(s)-\overline{y}(t)\|_{\infty} ds + \left( \left\lfloor\frac{t}{\Delta t}\right\rfloor \Delta t - t_0 \right) L_f (M_u + \Delta t + \Delta x) + \\
&+ 2 M_f \Delta t .
\end{align*}
By calling $L_I:=d C L_f$ and applying Gronwall's lemma we obtain the desired equality
\begin{align*}
\|y&(t) - \overline{y}(t)\|_{\infty} \leq \\
&\leq e^{L_I (\left\lfloor\frac{t}{\Delta t}\right\rfloor \Delta t-t_0)} \left( \left(\left\lfloor\frac{t}{\Delta t}\right\rfloor \Delta t-t_0\right) L_f (M_u + \Delta t + \Delta x) + 2 M_f \Delta t\right) \leq \\
&\leq  e^{L_I (t-t_0)} \left( (t-t_0) L_f (M_u + \Delta t + \Delta x) + 2 M_f \Delta t\right).
\end{align*}
\end{proof}

\begin{lemma} \label{th:lemma2}
Assume \ref{hp:1}, \ref{hp:2} and \ref{hp:3} hold true. If $\lambda > L_I$, then for $\Delta t \in \left(0,\frac{1}{2 \lambda}\right]$ there exist positive constants $C_1$, $C_2$ and $C_3$ such that
\[
\left\lvert J(x,u) - J^{\left\lfloor \frac{t}{\Delta t} \right\rfloor \Delta t}_{\Delta t,\Delta x}(x,\uline{u}) \right\lvert \leq C_1 \Delta t + C_2 \Delta x + C_3 M_u \qquad \forall x \in \overline{\Omega},
\]
where $u \in \mathcal{U}$ and $\uline{u} = (u_0,...,u_{N_t})$, where $u_n = u(t_n)$ for any $n=0,...,N_t$.
\end{lemma}

\begin{proof}
Fix $t\in(t_0,T)$, $x\in\Omega$ and $u\in\mathcal U$. Let $n:=\left\lfloor \dfrac{(t-t_0)}{\Delta t}\right\rfloor$
and define the discrete control sequence $\uline{u}=(u_0,\dots,u_N)$ by $u_k:=u(t_k)$.
Let $y$ be the solution of \eqref{eq:EDO_control} with control $u$ and initial condition $y(t)=x$. We denote by $\overline y$, as for Lemma 1, the piecewise constant extension of the numerical approximation $(y_0,...,y_{N_t})$ defined in \eqref{eq:EDO_discrete}.

We estimate the difference between the continuous cost functional $J$ defined in \eqref{eq:cost} and its discrete counterpart \eqref{eq:cost_discrete}. Writing the discrete functional in integral form on each time interval, we obtain
\begin{align*}
\bigg\lvert J&(x,u) - \left. J^{t_n}_{\Delta t,\Delta x}(x,\uline{u}) \right\rvert \leq \left\lvert \int_t^T g (y(s),s,u(s)) e^{-\lambda (s-t)} ds + e^{-\lambda (T-t)} \psi(y(T)) -\right. \\
&-\left. \Delta t \sum_{k=n}^{N_t-1} g(y_k,t_k,u_k) (1-\lambda \Delta t)^{k-n} + (1-\lambda \Delta t)^{N_t-n} \psi(y_{N_t})\right\rvert = \\
&= \left\lvert \int_t^T g (y(s),s,u(s)) e^{-\lambda (s-t)} ds + e^{-\lambda (T-t)} \psi(y(T)) -\right. \\
&-\left. \int_{t_n}^T \textbf{I}_1[g](\overline{y}(s),s,\overline{u}(s)) (1-\lambda \Delta t)^{\left\lfloor \frac{s-t_n}{\Delta t} \right\rfloor}ds + (1-\lambda \Delta t)^{N_t-n} \psi(y_{N_t})\right\rvert \leq \\
&\leq \int_{t_n}^T \lvert g(y(s),s,u(s)) - \textbf{I}_1[g](\overline{y}(s),s,\overline{u}(s))\rvert e^{-\lambda (s-t)} ds + \\
&+ \int_{t_n}^T \left\lvert\textbf{I}_1[g](\overline{y}(s),s,\overline{u}(s))\right\rvert \left\lvert e^{-\lambda (s-t)} - (1-\lambda \Delta t)^{\left\lfloor \frac{s-t_n}{\Delta t}\right\rfloor} \right\rvert ds + \\
&+ \int_{t_n}^t \lvert\textbf{I}_1[g](\overline{y}(s),s,\overline{u}(s))\rvert (1-\lambda \Delta t)^{\left\lfloor \frac{s-t_n}{\Delta t} \right\rfloor} ds + \\
&+ \left\lvert e^{-\lambda (T-t)} \psi(y(T)) - (1-\lambda \Delta t)^{N_t-n} \psi(\overline{y}(T)) \right\rvert = \\
&=: W + X + Y + Z.
\end{align*}

\medskip
\noindent
\textbf{Step 1. Estimate of $Y$.}
Since $g$ is bounded and the interpolation operator $\textbf{I}_1$ is linear, we can conclude that $\textbf{I}_1[g]$ is bounded, and so we have
\begin{align*}
Y &= \int_{t_n}^t \lvert\textbf{I}_1[g](\overline{y}(s),s,\overline{u}(s))\rvert (1-\lambda \Delta t)^{\left\lfloor \frac{s-t_n}{\Delta t} \right\rfloor} ds \leq \\
&\leq M_g \int_{t_n}^t (1-\lambda \Delta t)^{\left\lfloor \frac{s-t_n}{\Delta t} \right\rfloor} ds = \\
&= M_g (1-\lambda \Delta t)^{\left\lfloor \frac{t-t_n}{\Delta t} \right\rfloor} \left( t - t_n \right) \leq M_g (1-\lambda \Delta t)^{N_t} \Delta t \leq M_g \Delta t,
\end{align*}
where we used the fact that since $\frac{1}{2} \leq 1 - \lambda \Delta t < 1$, we have $(1-\lambda \Delta t)^{N_t} \leq 1$.

\medskip
\noindent
\textbf{Step 2. Estimate of $Z$ (terminal cost).}
To bound $Z$, we use that $\psi$ is Lipschitz continuous and Lemma \ref{th:lemma1}:
\begin{align*}
Z=\left\lvert e^{-\lambda (T-t)}\right. &\psi(y(T)) - \left.(1-\lambda \Delta t)^{N_t-\left\lfloor \frac{t}{\Delta t} \right\rfloor \Delta t} \psi(\overline{y}(T)) \right\rvert \leq \\
&\leq \max\{e^{-\lambda (T-t)},(1-\lambda \Delta t)^{N_t-\left\lfloor \frac{t}{\Delta t} \right\rfloor \Delta t}\} \lvert\psi(y(T)) - \psi(\overline{y}(T))\rvert \leq \\
&\leq C_4 \|y(T) - \overline{y}(T)\|_{\infty} \leq \\
&\leq C_4 e^{L_I T} \left( T L_f (M_u + \Delta t + \Delta x) + 2 M_f \Delta t\right) = \\
&=: K_1 M_u + K_2 \Delta x + K_3 \Delta t,
\end{align*}
where $C_4$ is a constant bounding $\max\{e^{-\lambda (T-t)},(1-\lambda \Delta t)^{N_t-\left\lfloor \frac{t}{\Delta t} \right\rfloor \Delta t}\}$ for all $t \in (t_0,T)$.

\medskip
\noindent
\textbf{Step 3. Estimate of $W$}
We decompose the integrand similarly to Lemma~1:

\begin{align*}
\lvert g(y(s),s,u(s)) &- \textbf{I}_1[g](\overline{y}(s),s,\overline{u}(s)) \rvert \leq \lvert g(y(s),s,u(s)) - g(y(s),s,\overline{u}(s))\rvert + \\
&+ \lvert g(y(s),s,\overline{u}(s)) - \textbf{I}_1[g](y(s),s,\overline{u}(s))\rvert + \\
&+ \lvert \textbf{I}_1[g](y(s),s,\overline{u}(s)) - \textbf{I}_1[g](\overline{y}(s),s,\overline{u}(s))\rvert \leq \\
&\leq L_g M_u + L_g \Delta x + d C L_g \|y(s) - \overline{y}(t)\|_{\infty} \leq \\
&\leq L_g M_u + L_g \Delta x + d C L_g e^{L_I s} \left( s L_f (M_u + \Delta t + \Delta x) + 2 M_f \Delta t \right),
\end{align*}
from which
\begin{align*}
W &= \int_{t_n}^T \lvert g(y(s),s,u(s)) - \textbf{I}_1[g](\overline{y}(s),s,\overline{u}(s))\rvert e^{-\lambda (s-t)} ds \leq \\
&\leq \int_{t_n}^T \left( L_g M_u + L_g \Delta x + \right. \\
&+\left. d C L_g e^{L_I s} \left( s L_f (M_u + \Delta t + \Delta x) + 2 M_f \Delta t \right) \right) e^{-\lambda (s-t)} ds = \\
&= C_1^{\prime} M_u + C_2^{\prime} \Delta x + d C L_g \int_{t_n}^T e^{(L_I - \lambda)s + \lambda t} \left( s L_f (M_u + \Delta t + \Delta x) + 2 M_f \Delta t \right) ds = \\
&= C_1^{\prime\prime} M_u + C_2^{\prime\prime} \Delta x + C_3^{\prime\prime} \Delta t.
\end{align*}

\medskip
\noindent
\textbf{Step 4. Estimate of $X$ (discount discretization).}

We use the boundedness of $\textbf{I}_1[g]$ and rewrite $(1-\lambda \Delta t)^{\left\lfloor \frac{s-t_n}{\Delta t}\right\rfloor}$ to get
\begin{align*}
X &= \int_{t_n}^T \lvert \textbf{I}_1[g](\overline{y}(s),s,\overline{u}(s)) \rvert \left\lvert e^{-\lambda (s-t)} - (1-\lambda \Delta t)^{\left\lfloor \frac{s-t_n}{\Delta t}\right\rfloor} \right\rvert ds \leq \\
&\leq M_g \int_{t_n}^T \left\lvert e^{-\lambda (s-t)} - (1-\lambda \Delta t)^{\left\lfloor \frac{s-t_n}{\Delta t} \right\rfloor} \right\rvert ds = \\
&= M_g \int_{t_n}^T \left\lvert e^{-\lambda (s-t)} - e^{\left\lfloor \frac{s-t_n}{\Delta t} \right\rfloor\log{(1-\lambda \Delta t)}} \right\rvert ds.
\end{align*}
Then, by applying Lagrange’s theorem and multiplying the last term inside the modulus by $\Delta t$, we obtain
\begin{align*}
X &\leq M_g \int_{t_n}^T e^{-\lambda \xi} \left\lvert \lambda (s-t) - \left\lfloor \frac{s-t_n}{\Delta t} \right\rfloor \lvert \log{(1-\lambda \Delta t)}\rvert \right\rvert ds = \\
&= M_g e^{-\lambda \xi} \int_{t_n}^T \lambda \left\lvert (s-t) - \left\lfloor \frac{s-t_n}{\Delta t} \right\rfloor \frac{\lvert \log{(1-\lambda \Delta t)}\rvert}{\lambda \Delta t} \Delta t \right\rvert ds = \\
&= M_g \lambda e^{-\lambda \xi} \int_{t_n}^T \left\lvert (s-t) - \left\lfloor \frac{s-t_n}{\Delta t} \right\rfloor \theta(\Delta t) \Delta t \right\rvert ds,
\end{align*}
where we have defined $\theta(\Delta t):=\frac{\lvert \log(1-\lambda \Delta t)\rvert }{\lambda \Delta t}$ (that is equal to $\frac{\log(1-\lambda \Delta t)}{\lambda \Delta t}$ since $1-\lambda \Delta t < 1$) and $\xi \in \left(t_n,T\right)$ is given by Lagrange's theorem. Differentiating it, we observe that $\theta$ is an increasing function of $\Delta t$ in the interval $\left(0,\frac{1}{2\lambda}\right)$. By de l'Hôpital's theorem, we get that $\theta(0) = 1$ (which is the minimum of $\theta$ in the chosen interval), and, by a direct computation, $\theta\left(\frac{1}{2\lambda}\right) = 2 \log(2)$ (which is the maximum of $\theta$ in the chosen interval). At this point we can bound
\[
\left\lvert (s-t_n)-\theta(\Delta t) \left\lfloor \frac{s-t_n}{\Delta t} \right\rfloor \Delta t \right\rvert \leq (\theta(\Delta t) -1) (s-t_n) + \Delta t,
\]
which is obtained recalling that $\left\lfloor \frac{s-t_n}{\Delta t} \right\rfloor \Delta t - \Delta t \leq s-t \leq \left\lfloor \frac{s-t_n}{\Delta t} \right\rfloor \Delta t + \Delta t$.
 Then we can write, using again the properties of positiveness of $\theta$,
\begin{align*}
X &\leq M_g \lambda e^{-\lambda \xi} \int_{t_n}^T \left\lvert (s-t) - \left\lfloor \frac{s-t_n}{\Delta t} \right\rfloor  \theta(\Delta t) \Delta t \right\rvert ds \leq\\
&\leq M_g \lambda e^{-\lambda \xi} \int_{t_n}^T \left( (\theta(\Delta t) -1) (s-t_n) + \Delta t \right) ds \leq \\
&\leq M_g \lambda e^{-\lambda \xi} \int_{t_n}^T \left( (\theta(\Delta t) -1) s + \Delta t \right) ds \leq \\
&\leq M_g \lambda e^{-\lambda \xi} \Delta t \left(\frac{\theta(\Delta t) - 1}{\Delta t} \frac{(T-t_0)^2}{2} + (T-t_0) \right).
\end{align*}
By another differentiation, we get that $\frac{\theta -1}{\Delta t}$ is an increasing function of $\Delta t$ in the interval $\left(0,\frac{1}{2\lambda}\right)$, and it attains its maximum in the point $\frac{1}{2\lambda}$, where its value is $2\lambda (2 \log(2)$, so we can write
\[
X \leq C_5 \Delta t.
\]
\medskip
\noindent
\textbf{Conclusion.}
Combining the estimates for $W,X,Y,Z$ we obtain
\[
\left\lvert J(x,t,u)-J^{t_n}_{\Delta t,\Delta x}(x,u)\right\rvert
\le C_1\Delta t + C_2\Delta x + C_3 M_u,
\]
which concludes the proof.
\end{proof}

\begin{theorem}\label{thm:conv}
Assume \ref{hp:1}, \ref{hp:2} and \ref{hp:3}. If $\lambda > L_I$, then for $\Delta t \in \left(0,\frac{1}{2 \lambda}\right]$ there exist positive constants $C_1$, $C_2$ and $C_3$ such that
\[
\|v-V\|_{\infty} \leq C_1 \Delta t + C_2 \Delta x + C_3 M_u.
\]
\end{theorem}

\begin{proof}
Let $x \in \overline{\Omega}$ and $n \in \{0,...,N_t\}$.

By theorem \ref{th:1}, there exists $\uline{u}^* = (u_0^*,...,u^*_{N_t}) \in U^{N_t+1}$ such that
\[
V(x,t_n) = J^{t_n}_{\Delta t,\Delta x}(x,\uline{u}^*) = \inf_{\uline{u} \in U^{N_t+1}} J^{t_n}_{\Delta t,\Delta x}(x,\uline{u}).
\]
Then
\[
v(x,t_n) - V(x,t_n) \leq J(x,t_n,u^*) - J^{t_n}_{\Delta t,\Delta x}(x,\uline{u}^*),
\]
where $u^* \in \mathcal{U}$ is such that $u^*(t_n) = u^*_n$ for any $ n=0,...,N_t$. From Lemma \ref{th:lemma2}
\[
v(x,t_n) - V(x,t_n) \leq C_1^{\prime}\Delta t + C_2^{\prime}\Delta x + C_3^{\prime} M_u .
\]

Let now $\tilde{u} \in \mathcal{U}$ be the control minimizing $J(x,t_n,\cdot)$ on $\mathcal{U}$, and $\uline{\tilde{u}} = (\tilde{u}_0,...,\tilde{u}_{N_t})$. Then, from Lemma \ref{th:lemma2},
\begin{align*}
V(x,t_n) - v(x,t_n) &\leq J^{t_n}_{\Delta t,\Delta x}(x,\uline{\tilde{u}}) - J(x,t_n,\tilde{u}) \leq \\
&\leq C_1^{\prime\prime} \Delta t + C_2^{\prime\prime}\Delta x + C_3^{\prime\prime} M_u .
\end{align*}
Let $C_{1,n}\Delta t + C_{2,n} \Delta x + C_{3,n} M_u := \max\{C_1^{\prime}\Delta t + C_2^{\prime}\Delta x + C_3^{\prime} M_u, C_1^{\prime\prime} \Delta t + C_2^{\prime\prime}\Delta x + C_3^{\prime\prime} M_u\}$. We get
\[
\lvert v(x,t_n) - V(x,t_n)\rvert \leq C_{1,n}\Delta t + C_{2,n} \Delta x + C_{3,n} M_u,
\]
and then, since $x \in \Omega$ is arbitrary,
\[
\|v(\cdot,t_n) - V(\cdot,t_n)\|_{\infty} \leq C_{1,n}\Delta t + C_{2,n} \Delta x + C_{3,n} M_u.
\]

We can repeat the computations above for each $n \in \{0,...,N_t-1\}$. Let $C_1\Delta t + C_2 \Delta x + C_3 M_u := \max_{n \in \{0,...,N_t-1\}} \{C_{1,n}\Delta t + C_{2,n} \Delta x + C_{3,n} M_u\} $. Then we have
\[
\|v-V\|_{\infty} \leq C_1\Delta t + C_2 \Delta x + C_3 M_u.
\]
For $n=N_t$ the approximation is exact.
\end{proof}

\section{Numerical experiments}\label{sec:4}

In this section we validate the theoretical results derived in Section \ref{sec:3} (Theorem \ref{thm:conv}) by computing
the Experimental Order of Convergence (EOC) of the SL scheme on two test problems. In both cases, we know the exact value function to compute the error. The optimization in \eqref{eq:SL} has been computed by comparison after the discretization of the control space.

The first example satisfies all the assumptions of the analysis, including the presence
of a positive discount factor. 

The second example considers an undiscounted problem, i.e. $\lambda = 0$, which falls
outside the theoretical framework developed in this paper. Although no convergence
result is proven in this case, the numerical experiment provides evidence that the
scheme still exhibits first order convergence, suggesting that the improved error
estimate may extend beyond the discounted setting.

In all experiments, we consider uniform refinements with $\Delta t = \Delta x$ and
measure the relative error with respect to the exact value function in the infinity vector norm $\|\cdot\|_{\infty}$.

\paragraph{Test 1.}

We consider the dynamics $f(y,t,u) = u$ in \eqref{eq:EDO_control} and consider the discounted finite horizon cost \eqref{eq:cost} with
\[
g(x,t,u) = \frac{1}{2} u^2, \qquad \psi(x) = \frac{1}{2}x^2, \qquad \lambda = 1.
\]
The exact value function for this control problem has the following structure
\begin{equation}\label{eq:V_y}
    v(t,x)=\frac{1}{2}P(t)\,x^2,
\end{equation}
where $P(t)$ solves a backward ODE of the form
\begin{align}
\left\{
\begin{aligned}
  -\dot{P}(s) &= P(s)  +  P^2(s), \qquad s\in[0,T]   \\
  P(T) &= 1,
\end{aligned}
\right.
\end{align}
whose solution is given by $P(t) = \frac{e^{T-t}}{2-e^{T-t}}$.

Finally, the optimal control for this problem is given by
\[
u^*(x(t))=-P(t)x(t).
\]

In our test, we will consider $x\in[-1,1],\ T =1 .$ For this purpose the control set will be $U =[-1,1].$ In Figure \ref{fig:exact1}, we show the contour lines of the exact value function in the left panel whereas the contour lines of the optimal control are given in the right plot. 


\begin{figure}[htbp]
    \centering
    \includegraphics[scale=0.3]{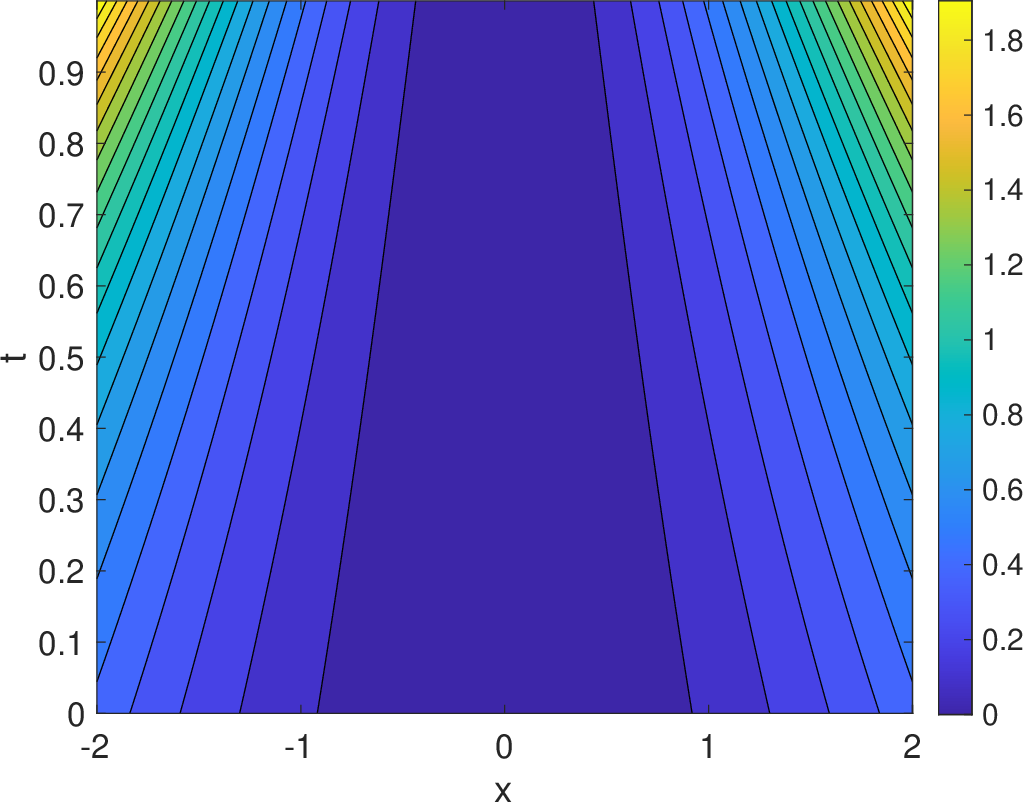}
    \includegraphics[scale=0.3]{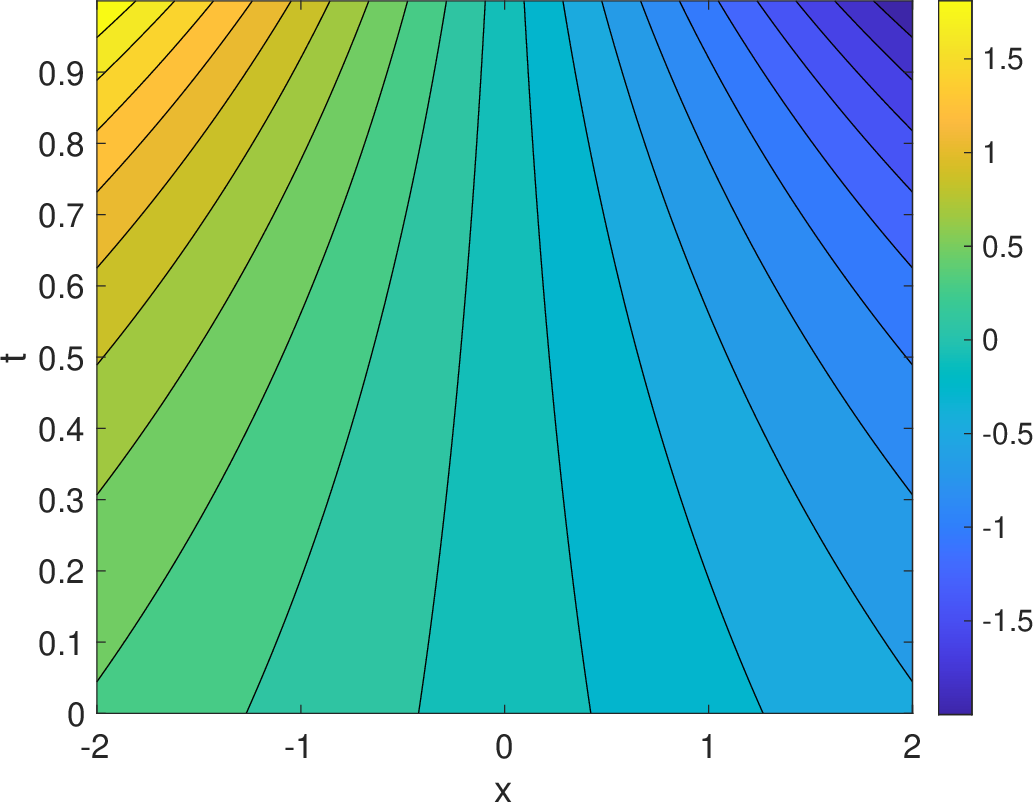}
    \caption{Test 1: Exact value function (left), optimal control (right).}
    \label{fig:exact1}
\end{figure}

In our numerical methods we set $\Delta x = \Delta t$ and we perform our tests for $\Delta x = \frac{0.1}{2^{i-1}}$ with $i=1,..., 5$. The control space was discretized keeping the same number of points of the spatial domains

The results are shown in Table \ref{tab:eoc1} and reflect the theoretical results from Theorem \ref{thm:conv}.

\begin{table}[h]
\centering
\begin{tabular}{|c|c|c|c|} \hline
$\Delta t$ & $\Delta x$ & error norm & EOC \\
\hline \hline
\phantom{0000}0.1 & \phantom{0000}0.1 & 0.0198 & - \\
\hline
\phantom{000}0.05 & \phantom{000}0.05 & 0.0100 &  0.9833    \\
\hline
\phantom{00}0.025 & \phantom{00}0.025 & 0.0025 &  0.9911    \\
\hline
\phantom{0}0.0125 & \phantom{0}0.0125 & 0.0013 &  0.9957  \\
\hline
0.00625 & 0.00625 & 0.00094 &  0.9979\\
\hline
\end{tabular}
\caption{Test 1: Error and EOC for the selected values of $\Delta t$ and $\Delta x$.}
\label{tab:eoc1}
\end{table}

For the sake of completeness we provide in the left panel of Figure \ref{fig1:exact1} the error convergence to help the reader to visualize it. In the right panel of the same figure we show the contour lines of the approximation value function computed with $\Delta x = 0.00625$ (see also left panel of Figure \ref{fig:exact1} for a comparison).

\begin{figure}[htbp]
    \centering
    \includegraphics[scale=0.3]{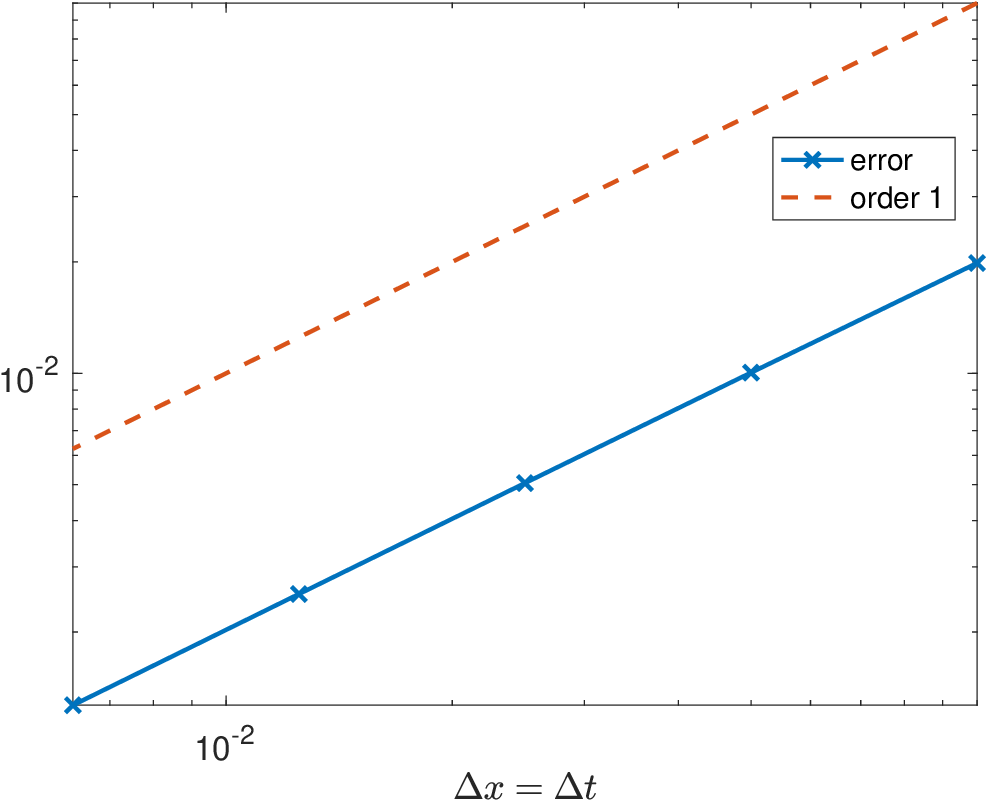}
    \includegraphics[scale=0.3]{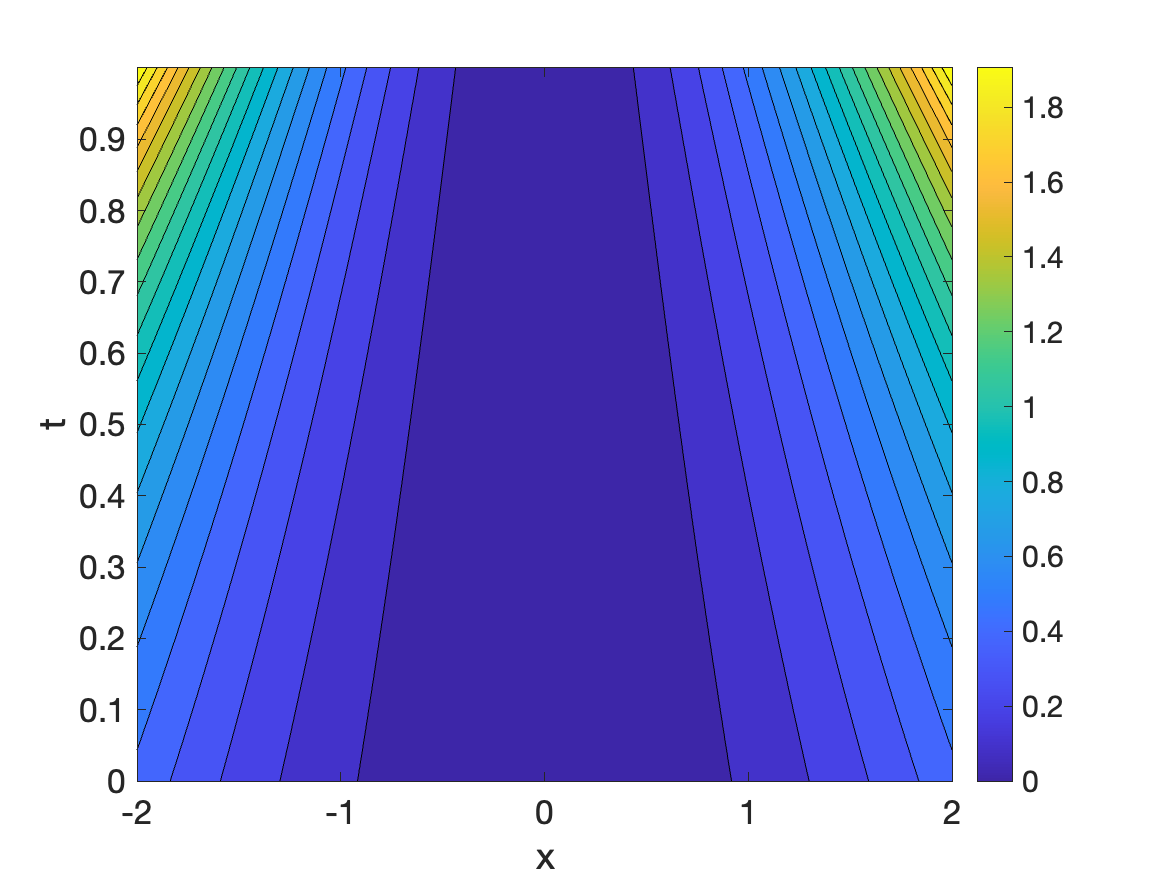}
    \caption{Test 1: Error convergence (left), approximated value function for $\Delta x = 0.00625$ (right).}
    \label{fig1:exact1}
\end{figure}


\paragraph{Test 2.}
In this example, we provide a test which goes beyond the assumption provided throughout the paper, indeed our discount factor will be $\lambda =0.$
This example is taken from \cite[Test 1]{AFS}. Here, we consider a two dimensional example and we will fix $\Omega = [-1,1]^2 \subset \mathbb{R}^2$, whose points we will denote $x = \begin{pmatrix} x_1 \\ x_2 \end{pmatrix}$, $U = [-1,1] \subset \mathbb{R}$, $t_0 = 0$ and $T=1$.

The dynamics, the running cost, final cost and discount factor are given by
\[
f(y,t,u) = \begin{pmatrix} u \\ y_1^2 \end{pmatrix},\qquad g(x,t,u) = 0, \qquad \psi(x) = -x_2, \qquad \lambda = 0,
\]
so the cost functional is given only by the final cost:

\[
J(x,t,u) =  -y_2(T),
\]
where $y(t) = (y_1(t), y_2(t))$ is the solution of \eqref{eq:EDO_control}. 

The exact value function for this problem reads
\[
v(x,t) = -x_2-x_1^2(T-t) - \frac{1}{3}(T-t)^3 - \lvert x_1\rvert (T-t)^2
\]
and it is shown in the left plot of Figure \ref{fig:exact2}.

\begin{figure}[htbp]
    \centering
    \includegraphics[scale=0.25]{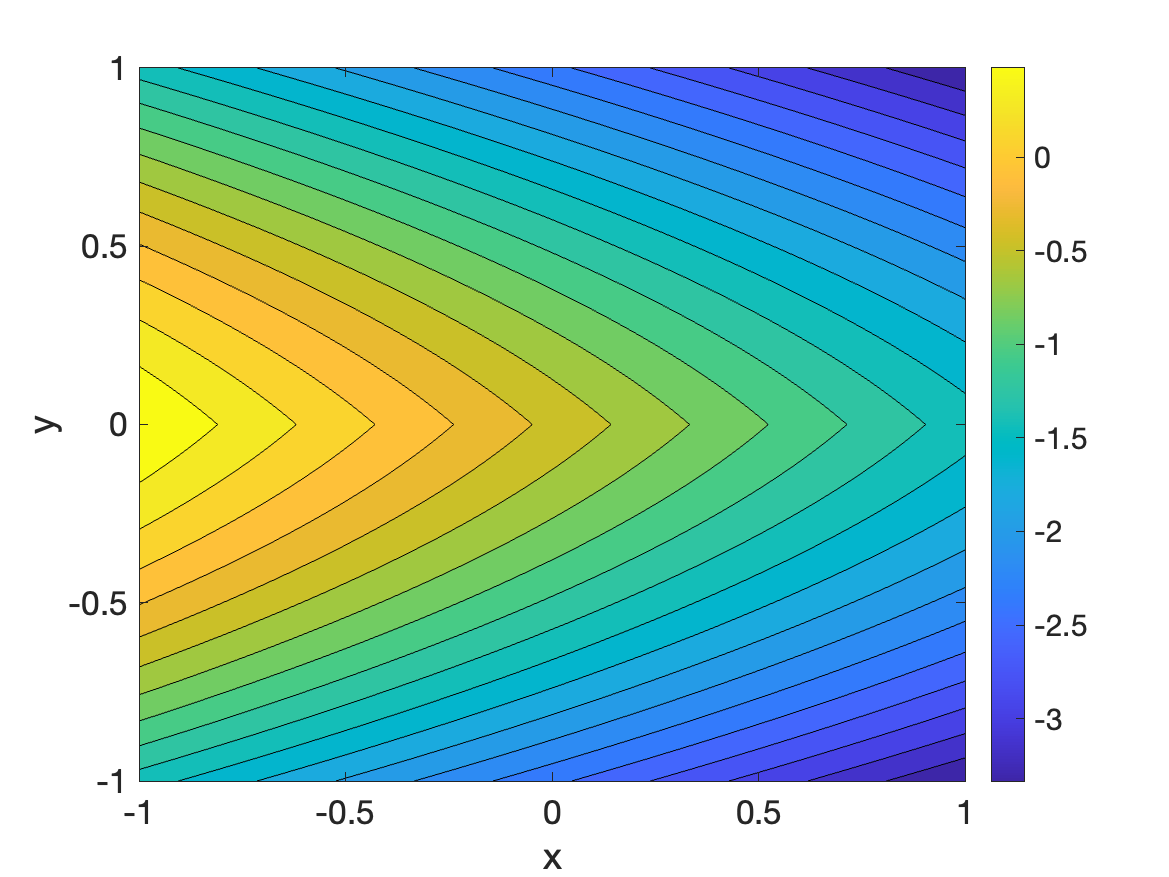}
    \includegraphics[scale=0.25]{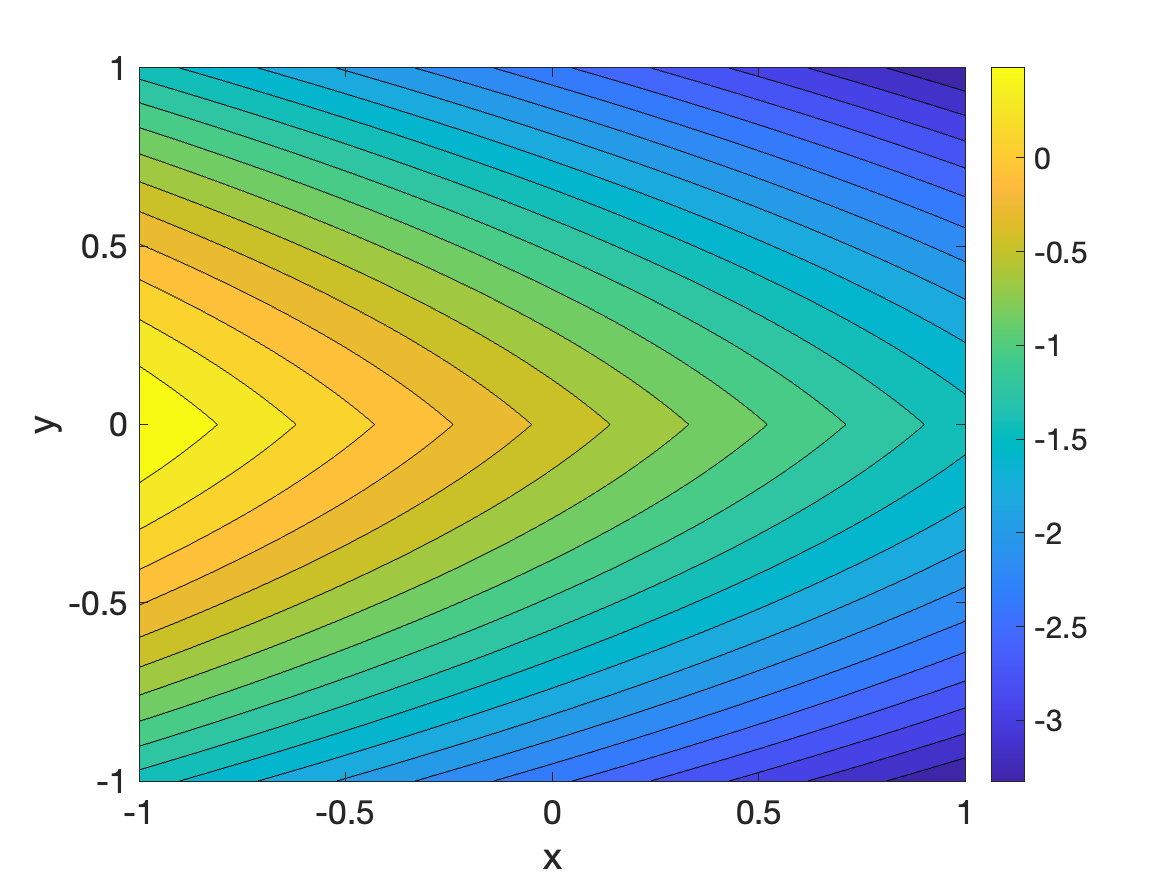}
    \caption{Test 2: Exact value function (left), approximate value function (right).}
    \label{fig:exact2}
\end{figure}


Our numerical error analysis is shown in Table \ref{tab:eoc2} and confirms the linear convergence both in space and time as for Theorem \ref{thm:conv}.


\begin{table}[h]
\centering
\begin{tabular}{|c|c|c|c|} \hline
$\Delta t$ & $\Delta x$ & error norm & EOC \\
\hline \hline
\phantom{0000}0.1 & \phantom{0000}0.1 & 0.1583 & - \\
\hline
\phantom{000}0.05 & \phantom{000}0.05 & 0.0771 & 1.0385 \\
\hline
\phantom{00}0.025 & \phantom{00}0.025 & 0.0380 & 1.0196 \\
\hline
\phantom{0}0.0125 & \phantom{0}0.0125 & 0.0189 & 1.0099\\
\hline
0.00625 & 0.00625 & 0.0094 & 1.0050\\
\hline
\end{tabular}
\caption{Test 2: Error and EOC for the selected values of $\Delta t$ and $\Delta x$.}
\label{tab:eoc2}
\end{table}


\begin{figure}
\centering
\includegraphics[scale=0.3]{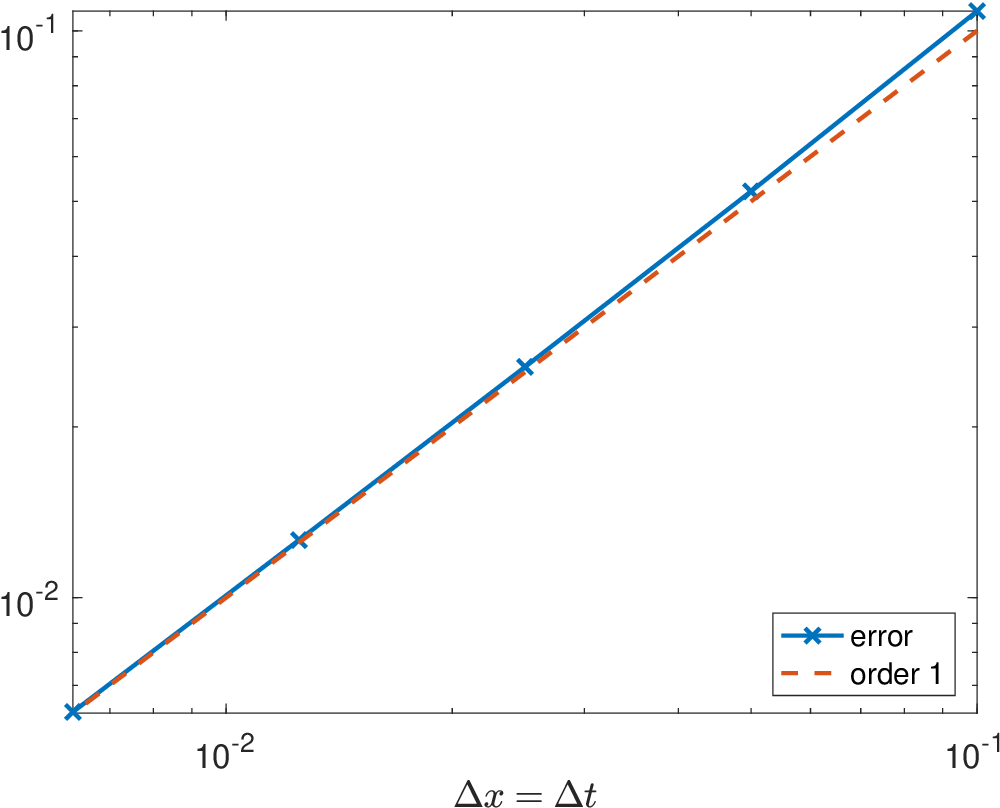}
\caption{Test 2: Experimental order of convergence for our example.}
\label{fig:eoc2}
\end{figure}

For completeness we also provide the contour lines of the approximate value function with $\Delta x = 0.0125 = \Delta t$ in the right plot of Figure \ref{fig:exact2}. The decay of the error and the EOC is finally shown in Figure \ref{fig:eoc2}.


\section{Conclusions}\label{sec:5}
In this work, we have analyzed the error for SL approximation schemes for HJB equations arising from finite horizon optimal control problems. By extending the techniques developed for the infinite horizon case, we derived improved error estimates that remove the mixed space–time term appearing in classical analyses. Under standard regularity assumptions on the dynamics, the running cost, and the final cost, we proved that the numerical value function converges linearly with respect to the time step and the spatial mesh size, in agreement with what is commonly observed in numerical simulations.

The numerical experiments confirm the theoretical findings and illustrate the sharpness of the new error bound. In particular, the test case with an explicit solution shows first-order convergence in both space and time, while an additional experiment in the undiscounted setting suggests that the same convergence behavior may persist even when the discount factor vanishes. This observation opens several directions for future research, including the extension of the present analysis to the case $\lambda = 0$, as well as to more general classes of control problems and discretization schemes.

\section*{Funding}


The authors are member of the GNCS group of INDAM. This work was partially by MUR-PRIN
Project 2022 PNRR (No. P2022JC95T) “Data-driven discovery and control of multi-scale interacting artificial agent systems”.

\printbibliography

@article{
DN,
author = {de Frutos, J. and Novo, J.},
title = {Optimal Bounds for Numerical Approximations of Infinite Horizon Problems Based on Dynamic Programming Approach},
journal = {SIAM Journal on Control and Optimization},
volume = {61},
number = {2},
pages = {415-433},
year = {2023},
}

@article{
F,
author = {Falcone, M.},
title = {A Numerical Approach to the Infinite Horizon Problem of Deterministic Control Theory},
journal = {Applied Mathematics and Optimization},
volume = {15},
pages = {1-13},
year = {1987}
}

@article{
Fec,
author = {Falcone, M.},
title = {Corrigenda: A Numerical Approach to the Infinite Horizon Problem of
Deterministic Control Theory},
journal = {Applied Mathematics and Optimization},
volume = {23},
pages = {213-214},
year = {1991}
}

@article{
AFS,
author = {Alla, A. and Falcone, M. and Saluzzi, L.},
title = {An Efficient DP Algorithm on a Tree-Structure for Finite Horizon Optimal Control Problems},
journal = {SIAM Journal on Scientific Computing},
volume = {41},
number = {4},
pages = {A2384-A2406},
year = {2019}
}

@incollection{
falconegiorgi,
author = {Falcone, M. and Giorgi, T.},
title = {An Approximation Scheme for Evolutive Hamilton-Jacobi Equations},
booktitle = {Stochastic Analysis, Control, Optimization and Applications},
publisher = {Birkhäuser},
address = {Boston},
pages = {289-303},
year = {1999}
}

@article{
CL,
author = {Crandall, M.G. and Lions, P.-L.},
title = {Viscosity solutions of Hamilton-Jacobi equations},
journal = {Transaction of the American Mathematical Society},
volume = {277},
number = {1},
pages = {1-42},
year = {1983}
}

@book{
bardicapuzzodolcetta,
author = {Bardi, M. and Capuzzo-Dolcetta, I.},
title = {Optimal Control and Viscosity Solutions of Hamilton-Jacobi-Bellman Equations},
publisher = {Birkhäuser},
address ={Boston},
year = {1997}
}

@book{
falconeferretti,
author = {Falcone, M. and Ferretti, R.},
title = {Semi-Lagrangian Approximation Schemes for Linear and Hamilton-Jacobi Equations},
publisher = {SIAM},
address ={Philadelphia},
year = {2014}
}

@article{
S,
author = {Sirovich, L.},
title = {Turbulence and the dynamics of coherent structures. Part I: Coherent structures},
journal = {Quarterly of Applied Mathematics},
volume = {45},
pages = {561-571},
year = {1987}
}

@article{
K,
author = {Kružkov, S.N.},
translator = {Smith, S.},
origlanguage = {russo},
title = {Generalized solutions of the Hamilton–Jacobi equations of eikonal type. I. Formulation of the problems; existence, uniqueness and stability theorems; some properties of the solutions},
journal = {Mathematics of the USSR-Sbornik},
volume = {27},
number = {3},
pages = {406-446},
year = {1975}
}

@book{
bellman,
author = {Bellman, R.},
title = {Dynamic Programming},
publisher = {Princeton University Press},
address ={Princeton},
year = {1957}
}


\end{document}